\documentclass[12pt]{amsart}

\usepackage[divide={1in,*,1in},hdivide={1in,*,1in}]{geometry}
\usepackage{amssymb,amsmath,amsthm}
\usepackage{pstricks,pst-node}
\usepackage{xspace}
\usepackage[%
	pdfauthor={Alissa S. Crans,Jozef H. Przytycki,Krzysztof K. Putyra},%
	pdftitle={Torsion in one-term distributive homology},colorlinks%
]{hyperref}

\definecolor{internalLink}{rgb}{0.5,0,0}
\definecolor{citeLink}{rgb}{0,0.5,0}
\definecolor{urlLink}{rgb}{0,0,0.5}

\hypersetup{linkcolor=internalLink}
\hypersetup{citecolor=citeLink}
\hypersetup{urlcolor=urlLink}

\theoremstyle{plain}
\newtheorem{theorem}{Theorem}[section]
\newtheorem{lemma}[theorem]{Lemma}
\newtheorem{conjecture}[theorem]{Conjecture}
\newtheorem{proposition}[theorem]{Proposition}
\newtheorem{corollary}[theorem]{Corollary}

\newtheorem*{FirstResult}{Theorem~\ref{thm:Hn-for-f-spindle-finite}}
\newtheorem*{SecondResult}{Theorem~\ref{thm:all-finite-groups}}

\theoremstyle{definition}
\newtheorem{definition}[theorem]{Definition}

\newtheorem{example}[theorem]{Example}

\theoremstyle{remark}
\newtheorem{remark}[theorem]{Remark}

\def\DeclareMathSymbol#1#2{%
	\newcommand#1{\mathrm{#2}}%
}

\DeclareMathOperator{\Z}{\mathbb{Z}}				
\DeclareMathOperator{\Q}{\mathbb{Q}}				
\DeclareMathSymbol{\id}{id}		
\DeclareMathSymbol{\rk}{rk}		
\DeclareMathOperator{\im}{im}		
\DeclareMathOperator{\coker}{coker}	
\DeclareMathSymbol{\cone}{cone}		

\makeatletter
\let\@osq[%
\let\@csq]%
\def\Cbefore#1#2#3{\@ifnextchar\@osq{\Cbefore@param{#1}{#2}}{\Cbefore@param{#1}{#2}[#3]}}%
\def\Cbefore@param#1#2[#3]{{}^{#3}{\Cmodifier{#1}}^{#2}}
\def\Cafter#1#2#3{\@ifnextchar\@osq{\Cafter@param{#1}{#2}}{\Cafter@param{#1}{#2}[#3]}}%
\def\Cafter@param#1#2[#3]{{\Cmodifier{#1}}^{#2,#3}}
\def\Cmodifier@reset{\let\Cmodifier\relax}
\def\Cmodifier@set#1{\def\Cmodifier##1{#1##1\Cmodifier@reset}}
\Cmodifier@reset
\def\Caugm{\Cmodifier@set\widetilde}
\def\Chat {\Cmodifier@set\widehat}
\makeatother
\newcommand{\CN}{\Cmodifier{C}^{N\!}}			
\newcommand{\CD}{\Cmodifier{C}^{D\!}}			
\newcommand{\CNxy}{\Cmodifier{C}^{N\!N\!}}	
\newcommand{\CNxx}{\Cmodifier{C}^{N\!D\!}}	
\newcommand{\Cb}{\Cafter{C}{N}{b}}		

\newcommand{\HN}{\Cmodifier{H}^{N\!}}			
\newcommand{\HD}{\Cmodifier{H}^{D\!}}			
\newcommand{\HNxy}{\Cmodifier{H}^{N\!N\!}}	
\newcommand{\Hb}{\Cafter{H}{N}{b}}		

\def\basecycle{\underline{\mathfrak{c}}}

\let\diff\partial
\DeclareMathOperator{\opn}{\star}			

\newcommand{\Orb}[1]{\bar #1}
\DeclareMathSymbol{\orb}{orb}
\DeclareMathSymbol{\init}{init}

\begin{document}

\title{Torsion in one-term distributive homology}

\author[Alissa S. Crans]{Alissa S. Crans}
\address{Department of Mathematics, Loyola Marymount University \\ Los Angeles, CA 90045}
\email{acrans@lmu.edu}

\author[J\'ozef H. Przytycki]{J\'ozef H. Przytycki}
\thanks{%
	JHP was partially supported by the NSA grant H98230-11-1-0175, by a~GWU-REF grant and by a~grant
	co-financed by the European Union (European Social Fund) and Greek national funds through
	the~Operational Program ``Education and Lifelong Learning'' of the~National Strategic Reference
	Framework (NSRF) --- Research Funding Program: ``Thales. Reinforcement of the~interdisciplinary
	and/or inter-institutional research and innovation.''
}
\address{%
	Department of Mathematics, George Washington University\\ Washington, DC 20052\\
	and Institute of Mathematics, University of Gda\'nsk, Poland
}
\email{przytyck@gwu.edu}

\author[Krzysztof K. Putyra]{Krzysztof K. Putyra}
\thanks{%
	KKP was partially supported by the~Columbia University topology RTG grant DMS-0739392.
}
\address{Department of Mathematics, Columbia University\\ New York, NY 10027}
\email{putyra@math.columbia.edu}

\begin{abstract}
	The one-term distributive homology was introduced in \cite{Prz-distr-survey} as an~atomic
	replacement of rack and quandle homology, which was first introduced and developed by
	Fenn-Rourke-Sanderson \cite{FennRourkeSand} and Carter-Kamada-Saito \cite{CarterKamadaSaito}.
	This homology was initially suspected to be torsion-free \cite{Prz-distr-survey}, but we show
	in this paper that the~one-term homology of a~finite spindle can have torsion. We carefully
	analyze spindles of block decomposition of type $(n,1)$ and introduce various techniques to
	compute their homology precisely. In addition, we show that any finite group can appear as
	the~torsion subgroup of the~first homology of some finite spindle. Finally, we show that if
	a~shelf satisfies a~certain, rather general, condition then the one-term homology is trivial
	--- this answers a~conjecture from \cite{Prz-distr-survey} affirmatively.
\end{abstract}

\maketitle

\setcounter{tocdepth}{1}
\tableofcontents

\section{Introduction}\label{sec:intro}
For any set $X$, we can consider colorings of arcs of a~link diagram by elements of $X$. Motivated by
a~Wirtingen presentation of the~fundamental group of a~link complement, we may assume that
overcrossings preserve colors while undercrossings change them in a~way described by some binary
operation $\opn\colon X\times X\to X$, as shown in Fig.~\ref{fig:coloring}.

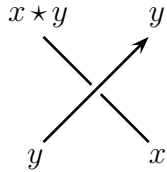
\begin{figure}[ht]
	\begin{center}
		\psset{unit=0.7cm}
		\begin{pspicture}(0,-0.2)(2,2.2)
			\psline[linewidth=1pt](2,0)(0,2)
			\psline[linewidth=5pt,linecolor=white](0,0)(2,2)
			\psline[linewidth=1pt,arrowsize=5pt]{->}(0,0)(2,2)
			\rput[Bl](2,-8pt){$x$}\rput[B ](-2pt,2.3){$x\opn y$}
			\rput[Br](0,-8pt){$y$}\rput[Bl](2,2.3){$y$}
		\end{pspicture}
	\end{center}
	\caption{Propagation of colors at a~crossing}\label{fig:coloring}
\end{figure}

The~requirement that the~Reidemeister moves change the~coloring only locally results in~several conditions
on $(X,\opn)$, making it a~quandle \cite{Joyce-quandle} or a~rack \cite{FennRourke-rack-first}.
However, the~most important is the~third Reidemeister move, visualized in Fig.~\ref{fig:R-3}, because
of its close connection to the~Yang-Baxter equation \cite{CES, Eis, Prz-distr-survey}. This requires $\opn$ to be
distributive, i.e.~$(x\opn y)\opn z = (x\opn z)\opn (y\opn z)$, and pairs $(X,\opn)$ satisfying this
condition are called \emph{shelves}. If $\opn$ is also idempotent, i.e. $x\opn x=x$, $(X,\opn)$ is
a~\emph{spindle} \cite{Crans-2-algebras}.

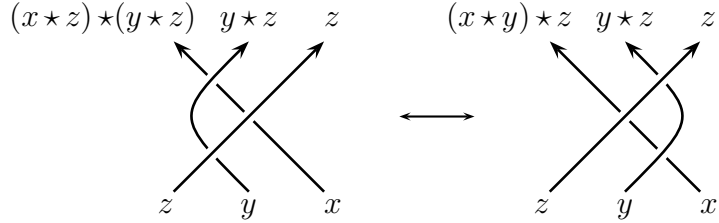
\begin{figure}
	\begin{center}
		\begin{pspicture}(0,-0.1)(7,2.4)
			\rput[bl](0,0){
				\psline[linewidth=1pt,arrowsize=5pt]{->}(2,0)(0,2)
				\psbezier[linewidth=5pt,linecolor=white](1,0)(0,1)(0,1)(1,2)
				\psbezier[linewidth=1pt,arrowsize=5pt]{->}(1,0)(0,1)(0,1)(1,2)
				\psline[linewidth=5pt,linecolor=white](0,0)(2,2)
				\psline[linewidth=1pt,arrowsize=5pt]{->}(0,0)(2,2)
				\rput[Bl](2,-8pt){$x$}\rput[Br](0.3,2.2){$(x\opn z)\opn(y\opn z)$}
				\rput[B ](1,-8pt){$y$}\rput[B ](1.0,2.2){$\phantom{(}y\opn z$\phantom{)}}
				\rput[Br](0,-8pt){$z$}\rput[Bl](2.0,2.2){$z$}
			}
			\psline{<->}(3,1)(4,1)
			\rput[bl](5,0){
				\psline[linewidth=1pt,arrowsize=5pt]{->}(2,0)(0,2)
				\psbezier[linewidth=5pt,linecolor=white](1,0)(2,1)(2,1)(1,2)
				\psbezier[linewidth=1pt,arrowsize=5pt]{->}(1,0)(2,1)(2,1)(1,2)
				\psline[linewidth=5pt,linecolor=white](0,0)(2,2)
				\psline[linewidth=1pt,arrowsize=5pt]{->}(0,0)(2,2)
				\rput[Bl](2,-8pt){$x$}\rput[Br](0.3,2.2){$(x\opn y)\opn z$}
				\rput[B ](1,-8pt){$y$}\rput[B ](1.0,2.2){$\phantom{(}y\opn z$\phantom{)}}
				\rput[Br](0,-8pt){$z$}\rput[Bl](2.0,2.2){$z$}
			}
		\end{pspicture}
	\end{center}
	\caption{Third Reidemeister move forces $\opn$ to be distributive}\label{fig:R-3}
\end{figure}

Link invariants come not only from counting colorings by rack or quandles, but also from their
homologies, see \cite{Qndl-cocycle-invs, Qndl-homology}. We noticed in \cite{Prz-distr-survey, 
PrzSik-distr-hom} that homology groups can be defined similarly for any~shelf or spindle. Even 
more, there is a~chain complex with a~simpler differential, called a~\emph{one-term distributive
chain complex} $C^{\opn}(X)$ (see Section~\ref{sec:definitions} for a~definition). We showed in
\cite{Prz-distr-survey,PrzSik-distr-hom} that if $(X,\opn)$ is a~rack, then $C^{\opn}(X)$ is acyclic.
More generally, to force $C^{\opn}(X)$ to be acyclic it is enough to have just one element $y\in X$ such that $x\mapsto x\opn y$ is
a~bijection. This is perhaps the~reason why this homology has
never been examined before. At first, one would be tempted to suspect that $H^{\opn}(X)$ is always trivial,
but we quickly computed the homology for a~right trivial shelf $(X, \dashv)$, where $a\dashv y=y$,
and found it to be a~large free group \cite{PrzSik-distr-hom}. For a~while all one-term homology we
computed was free; only in February of 2012 did we find two four-element spindles with torsion
in homology. More precisely, our examples are given by the~following tables:
\begin{center}
	\vskip 0.5\baselineskip
	\begin{tabular}{c|cccc}
		\raisebox{1pt}{$\opn_1$} & 1 & 2 & 3 & 4 \\
		\hline
		1 & 1 & 2 & 3 & 4 \\
		2 & 1 & 2 & 3 & 4 \\
		3 & 1 & 2 & 3 & 4 \\
		4 & 2 & 1 & 1 & 4
	\end{tabular}
	\qquad
	\begin{tabular}{c|cccc}
		\raisebox{1pt}{$\opn_2$} & 1 & 2 & 3 & 4 \\
		\hline
		1 & 1 & 2 & 4 & 3 \\
		2 & 1 & 2 & 4 & 3 \\
		3 & 2 & 1 & 3 & 4 \\
		4 & 2 & 1 & 3 & 4
	\end{tabular}
	\vskip 0.5\baselineskip
\end{center}
Using \texttt{Mathematica}, we found that the~first homology for both spindles has
$\Z_2$-torsion. Namely, we obtained the~following groups:
\begin{align*}
	H_0^{\opn_1}(X) &= \Z^2,							& H_0^{\opn_2}(X) &= \Z^2,							\\
	H_1^{\opn_1}(X) &= \Z^2\oplus\Z_2,		& H_1^{\opn_2}(X) &= \Z^2\oplus\Z_2^4,	\\
	H_2^{\opn_1}(X) &= \Z^8\oplus\Z_2^4,	& H_2^{\opn_2}(X) &= \Z^8\oplus\Z_2^{12}.
\end{align*}
In this paper, we compute the~homology of the~first spindle and, more generally, of other \hbox{$f$-\emph{spindles}},
which are spindles given by a~function $f\colon X_0\to X_0$ where $X=X_0\sqcup\{b\}$ and $x\opn y=y$,
unless $x=b$, in which case $b\opn y = f(y)$ (see Definition~\ref{def:f-spindle}). This family of
spindles was introduced in~\cite{PrzSik-distr-hom}. If $X$ is finite, we prove in Section~\ref{sec:homology}
the~following formulas for normalized homology (see Section~\ref{sec:definitions} for a~definition
of a~normalized complex):

\begin{FirstResult}
	Assume $X$ is a~finite $f$-spindle. Then its homology is given by the~formulas
	\begin{equation*}
	\setlength\arraycolsep{2pt}
	\left\{\begin{array}{rl}
			\Caugm\HN_0(X) &= \Z^{\orb(f)},\\
						\HN_1(X) &= \Z^{(\orb(f)-1)|X_0|+2\orb(f)} \oplus
												\Z_\ell^{\init(f)}, \phantom{\Big|^{(|X|-1)^n}}\\
						\HN_n(X) &= \left(\Z^{(\orb(f)-1)|X|^2+|X|} \oplus
												\Z_\ell^{\init(f)|X|}\right)^{\oplus(|X|-1)^{n-2}},\quad\textrm{for }n\geqslant 2.
	\end{array}\right.
	\end{equation*}
	In particular, $\HN_{n+1}(X) = \HN_n(X)^{\oplus(|X|-1)}$ for $n\geqslant 2$.
\end{FirstResult}

Here, $\orb(f)$ and $\init(f)$ stand, respectively, for the~number of orbits of $f$ and the~number
of elements that are not in the~image of $f$. This shows that any power of a~cyclic group can
appear as the~torsion subgroup of $H_1(X)$ for some spindle. The~other finite abelian groups are
realized by \emph{block spindles}, defined in Section~\ref{sec:odds-ends}. The~idea is that we take
several blocks $X_i$ and a~function $f_i\colon X_i\to X_i$ for each of them, and we take as $X$
their disjoint sum together with a~one-element block $\{b\}$. Then each $X_i^+ := X_i\sqcup\{b\}$
is a~subspindle, which contributes some torsion to $H_1(X)$. We show that, in fact, there is no
more torsion.

\begin{SecondResult}
	Assume a~block spindle $X$ has a~one-element block $\{b\}$. Then
	\begin{equation*}
		H_1(X) \cong F\oplus\bigoplus_{i\in I}H_1(X_i^+),
	\end{equation*}
	where $F$ is a~free abelian group of rank $\sum_{i\neq j} orb(f_i)|X_j| $. In~particular, every
	finite abelian group can be realized as the~torsion subgroup of $H_1(X)$ for some spindle $X$.
\end{SecondResult}

This paper is arranged as follows. We provide basic definitions in Section~\ref{sec:definitions},
including the construction of a~distributive chain complex and its variants: augmented, reduced, and
related chain complexes. We also include a~discussion about degenerate and normalized complexes and
how they are related to each other.

The~next two sections are devoted to the~calculation of homology groups for $f$-spindles. In
Section~\ref{sec:torsion} we define an~$f$-spindle, provide a~few examples, and then compute
the~first homology group. Then in Section~\ref{sec:homology} we generalize these calculations
for any homology groups. We conclude this section with a~presentation of homology groups in terms
of generators and relations for any $f$-spindle, not necessarily finite.

The~final section is split into four parts. In the~first, we give
a~presentation of the~relative homology groups with respect to the~subspindle $X_0\subset X$. The~second
part contains a~proof of Theorem~\ref{thm:all-finite-groups} and the~third discusses the~Growth
Conjecture from~\cite{PrzSik-distr-hom}. The~last part contains a~result about the acyclicity of
a~distributive chain complex under a~small condition --- all that was known previously was that
homology was annihilated by some number, leaving it with a~possibility to have torsion
\cite{Prz-distr-survey}.

\section{Distributive homology}\label{sec:definitions}
A~spindle $(X,\opn)$ consists of a~set $X$ equipped with a~binary operation $\opn\colon X\times X\to X$
that~is
\begin{enumerate}
	\item idempotent, $x\opn x = x$, and
	\item self-distributive, $(x\opn y)\opn z = (x\opn z)\opn (y\opn z)$.
\end{enumerate}
A~\emph{(one-term) distributive chain complex} $C^{\opn}(X)$ of $X$ is defined as follows (see also
\cite{Prz-distr-survey,PrzSik-distr-hom}):
\begin{align}
	C^{\opn}_n(X)&:=\Z\! X^{n+1} = \Z\langle (x_0,\dots,x_n)\ |\ x_i\in X\rangle, \\
	\diff_n &:= \sum_{i=0}^n (-1)^i d^i,
\end{align}
where maps $d^i$ are given by the~formulas
\begin{align}
	d^0(x_0,...,x_n) &= (x_1,...,x_n), \textrm{ and}\\
	d^i(x_0,...,x_n) &= (x_0\opn x_i,...,x_{i-1}\opn x_i,x_{i+1},...,x_n).
\end{align}
We check that $d^id^j = d^{j-1}d^i$ whenever $i<j$, which implies $\diff^2=0$.
The~homology of this chain complex is called the~\emph{(one-term) distributive homology} of $(X,\opn)$
and it will be denoted by $H^{\opn}(X)$. There is also an~\emph{augmented} version, $\widetilde C(X)$,
with $\widetilde C^{\opn}_n(X) = C^{\opn}_n(X)$ for $n\geqslant 0$, but $\widetilde C^{\opn}_{-1}(X)=\Z$
and $\diff_0(x)=1$. Its homology, called the~\emph{augmented distributive homology} $\widetilde H^{\opn}(X)$,
satisfies the following, as in the~classical case:
\begin{equation}\label{eq:augm-H-vs-H}
	H^{\opn}_n(X)=
			\begin{cases}
					\Z \oplus \widetilde H^{\opn}_n(X),  & n = 0, \\
					\widetilde H^{\opn}_n(X),            & n > 0.
			\end{cases}
\end{equation}
For simplicity, we will omit $\opn$ and write $C(X)$ and $H(X)$ for the~distributive chain complex
and its homology, and similarly for the~augmented versions. Furthermore, we will use the~shorthand
notation $\underline x := (x_0,\dots,x_n)$ for a~sequence of elements and ocassionally
a~multilinear notation\footnote{
	Think of $(x_0,\dots,x_n)$ as an~element $x_0\otimes\cdots\otimes x_n$ in $\Z\!X^{\otimes(n+1)}$.
}
$(\ldots,x_i+x_i',\ldots) := (\ldots,x_i,\ldots)+(\ldots,x_i',\ldots)$.
In particular, $(0,\underline x) = 0$.	

Assume $Y\subset X$ is a~subspindle of $X$, i.e. $x\opn y\in Y$ whenever $x,y\in Y$. It follows
directly from the~definition above that the~chain complex $C(Y)$ is a~subcomplex of $C(X)$.
The~quotient $C(X,Y):=C(X)/C(Y)$ is called the~\emph{relative chain complex} of $X$ modulo $Y$. It is
spanned by sequences $\underline x$ where not all entries are from $Y$. Clearly, there is a~long exact
sequence of homology
\begin{equation}
	\ldots\longrightarrow H_n(Y)
				\longrightarrow H_n(X)
				\longrightarrow H_n(X,Y)
				\longrightarrow H_{n-1}(Y)
				\longrightarrow\ldots
\end{equation}
and an~analogous sequence when we replace the homologies of $Y$ and $X$ with their augmented versions.

Let $f\colon X\to Y$ be a~homomorphism of spindles, i.e. $f(x\opn x') = f(x)\opn f(x')$. There is
an~induced chain map $f_\sharp\colon C(X)\to C(Y)$ sending a~sequence $(x_0,\dots,x_n)$ to
$(f(x_0),\dots,f(x_n))$. In the case where $r\colon X\to X$ is a~rectraction on a~subspindle $Y$
(i.e. $r(X)=Y$ and $r|_Y = \id$), one has a~decomposition $C(X)\cong C(Y)\oplus C(X,Y)$.
In particular, for any element $b\in X$ one has $C(X)\cong C(b)\oplus C(X,b)$, so that $C(X,b)$
is independent of the~choice of $b$. It is called the~\emph{reduced chain complex} (see~\cite{PrzPut-lattices}).
As a~subcomplex of $C(X)$, it is generated by differences $\underline x - \underline b$.

Idempotency of the~spindle operation in $X$ implies that its~distributive chain complex $C(X)$ is in fact
a~\emph{weak simplicial module} (see~\cite{Prz-distr-survey,PrzPut-lattices}). In particular, there
are notions of degenerate and normalized complexes. Indeed, if $\underline x$ has a~repetition, say
$x_i=x_{i+1}$, so does each entry in $\diff\underline x$, as $d^i\underline x = d^{i+1}\underline x$
cancels each other and other faces preserve the~repetition. Hence, sequences with repetition span
a~subcomplex $\CD(X)\subset C(X)$, called the~\emph{degenerate complex} of $X$. Explicitly,
\begin{equation}
	\CD_n(X) := \Z\langle\ \underline x\ \vert\ x_i=x_{i+1}\textrm{ for some }0\leqslant i<n\ \rangle.
\end{equation}
The~quotient $\CN(X):=C(X)/\CD(X)$ is called the~\emph{normalized complex} and is generated by sequences
with no repetitions. Degenerate and normalized homology are written, respectively, as $\HD(X)$ and $\HN(X)$.
In classical homology theories (simplicial homology, group homology, etc.) the~degenerate complex is
acyclic, so that $\HN \cong H$. However, this does not hold for a~weak simplicial module and we can
have nontrivial degenerate homology in the~distributive case, so that normalized homology $\HN(X)$
is usually different from $H(X)$. However, we can split the~degenerate complex apart. This was first
shown in \cite{LithNelson-splitting} for quandles (for the~two-term variant of distributive homology)
and an~explicit formula for the~splitting map appeared for the~first time in \cite{NP-dihedral}.
It was observered in \cite{Prz-distr-survey,PrzPut-lattices} that the~same map works for the~one-term
variant as well.

\begin{theorem}[cf. \cite{Prz-distr-survey,PrzPut-lattices}]
	Let $(X,\opn)$ be a~spindle. Then the~exact sequence of complexes
	\begin{equation}
		0	\longrightarrow \CD(X)
			\longrightarrow C(X)
			\longrightarrow \CN(X)
			\longrightarrow 0
	\end{equation}
	splits. In particular, $H(X)\cong \HN(X)\oplus\HD(X)$.
\end{theorem}

\begin{example}
	A~normalized complex for a~one-element splindle $\{b\}$ has a~unique generator in degree 0.
	Since a~retraction splits a~normalized complex as well, we obtain an~isomorphism
	$\Caugm\HN(X)\cong\HN(X,b)$ for any $b\in X$, so that the~normalized versions of reduced and
	augmented homologies coincide. In fact, the~inclusion $\CN(X,b)\subset\Caugm\CN(X)$ is
	a~homotopy equivalence.
\end{example}

In \cite{PrzPut-degen} we canonically decomposed the~degenerate complex into a~bunch of copies of
the~normalized complex. Therefore, normalized homology carries all information and there is no need
to bother with the~degenerate part.

\begin{theorem}[cf. \cite{PrzPut-degen}]\label{thm:CD-from-CN}
	Let $(X,\opn)$ be a~spindle. Then the~degenerate complex decomposes as
	\begin{equation}\label{eq:CD-from-CN}
		\CD_n(X) = \bigoplus_{p+q=n-2}\! \widetilde C_p(X)\otimes\CN_q(X)
	\end{equation}
	with the~differential acting only on the~first factor: $\diff(\underline x\otimes\underline y) =
	\diff\underline x\otimes\underline y$.
\end{theorem}

\noindent
In particular, $\HD_0(X)=\HD_1(X)=0$ and $\HD_2(X)=\widetilde H_0(X)\otimes\Z\! X$.

\section{A~family of spindles with torsion}\label{sec:torsion}
In this section we construct a family of spindles that have torsion in their homology groups.
Namely, we can realize every power of a~cyclic group as a~torsion subgroup of $H_1$.

\begin{definition}\label{def:f-spindle}
	Choose a~set together with a~basepoint, $(X,b)$, and set $X_0 = X-\{b\}$. Any function
	$f\colon X_0\to X_0$ induces a~spindle on $X$ by defining
	\begin{equation}
		x\opn y = \begin{cases}
									f(y), & \textrm{if }x=p,\\
										y , & \textrm{if }x\neq p.
							\end{cases}
	\end{equation}
	We call $(X,\opn)$ an~\emph{$f$-spindle} and denote it by $X_f$.
\end{definition}

The~function $f$ induces a~discrete semi-dynamical system on $X_0$. We can visualize it as a~graph
$\Gamma_f$ whose vertices are elements of $X_0$ and with directed edges $x\rightarrow f(x)$. Every
vertex in this graph has exactly one outcoming edge. If a~vertex $v$ has no incoming edges, it is
called an~\emph{initial} vertex or a~\emph{source}. The initial vertices are precisely the elements
of $X_0$ that are not in the~image of $f$. The~number of such elements will be denoted by $\init(f)$.
Finally, connected components of $\Gamma_f$ correspond to orbits of the~semi-dynamical system
induced by $f$. Their number will be denoted by $\orb(f)$. The~orbit of an~element $x$ will be
written as $\Orb{x}$.

Consider a~connected component $\Gamma^0_f$ of $\Gamma_f$. It can either be an~infinite directed tree
with no loops (so that $f^i(x)\neq x$ for any $i>0$) or there exists a~number $k>0$ such that for any
vertex $v\in\Gamma^0_f$ we have $f^{i+k}(v) = f^i(v)$ for $i$ big enough. When we choose the smallest
such $k$, then the set $\{f^i(v),\dots,f^{i+k-1}(v)\}$ is a~unique cycle in $\Gamma^0_f$, which we
call a~\emph{soma} of $\Gamma^0_f$. Clearly, the~component $\Gamma^0_f$ consists of this cycle and
\emph{dendrites}, possible infinite, as can be seen in Fig.~\ref{fig:soma-and-dendrites}.

Finally, we choose a~single vertex $v^i$ from any component of $\Gamma_f$ and set $\ell$ to be
the~greatest common divisor of lengths of all cycles in $\Gamma_f$. If $\Gamma_f$ has no
cycles at all, set $\ell=0$.

\begin{figure}[t]
	\begin{pspicture}(4.1,3.0)
		\psset{unit=0.7cm,nodesep=1mm,linewidth=0.5pt,arrowsize=5pt}%
		\rput(0.2,2.9){\rnode{n1}{\qdisk(0,0){1pt}}}\rput(0.2,1.1){\rnode{n2}{\qdisk(0,0){1pt}}}
		\rput(1.7,0.2){\rnode{n3}{\qdisk(0,0){1pt}}}\rput(3.2,1.1){\rnode{n4}{\qdisk(0,0){1pt}}}
		\rput(3.2,2.9){\rnode{n5}{\qdisk(0,0){1pt}}}\rput(1.7,3.8){\rnode{n6}{\qdisk(0,0){1pt}}}
		\ncline{->}{n1}{n2}\ncline{->}{n2}{n3}\ncline{->}{n3}{n4}
		\ncline{->}{n4}{n5}\ncline{->}{n5}{n6}\ncline{->}{n6}{n1}
		\rput(4.8,3.5){\rnode{u1}{\qdisk(0,0){1pt}}}
		\rput(5.8,4.5){\rnode{u2}{\qdisk(0,0){1pt}}}
		\rput(6.2,2.9){\rnode{u3}{\qdisk(0,0){1pt}}}
		\ncline{<-}{n5}{u1}\ncline{<-}{u1}{u2}\ncline{<-}{u1}{u3}
		\rput(4.9,1.0){\rnode{d1}{\qdisk(0,0){1pt}}}
		\rput(6.7,1.9){\rnode{d2}{\qdisk(0,0){1pt}}}
		\rput(6.4,0.3){\rnode{d3}{\qdisk(0,0){1pt}}}
		\rput(8.2,0.6){\rnode{d4}{\qdisk(0,0){1pt}}}
		\ncline{<-}{n4}{d1}\ncline{<-}{d1}{d2}\ncline{<-}{d1}{d3}\ncline{<-}{d3}{d4}
		\rput(-1.3,3.7){\rnode{l1}{\qdisk(0,0){1pt}}}
		\rput(-1.6,2.4){\rnode{l2}{\qdisk(0,0){1pt}}}
		\ncline{<-}{n1}{l1}\ncline{<-}{n1}{l2}
	\end{pspicture}
	\caption{A~typical connected component of $\Gamma_f$. It has four dendrites and six initial vertices.}
	\label{fig:soma-and-dendrites}
\end{figure}
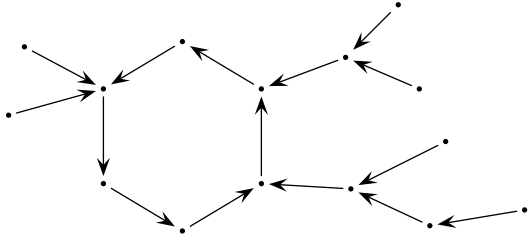

\begin{example}\label{ex:sigma-spindle}
	Let $X=\{0,\dots,k+1\}$ for some $k\geqslant 1$ and set $b=0$ so~that $X_0=\{1,\dots,k+1\}$.
	Define $\sigma_k\colon X_0\to X_0$ as follows:
	\begin{equation}
		\sigma_k(n) := \begin{cases}
			n+1,&\textrm{if } n<k,\\
			1,&\textrm{if } n=k,k+1
		\end{cases}
	\end{equation}
	The~graph for $\sigma_5$ is shown in Fig.~\ref{fig:sigma-graph}. It has one component with
	a~cycle of length $k=5$ and a~unique initial vertex.
\end{example}

\noindent
It appears that the first homology group of the~spindle obtained from $\sigma_k$ has $\Z_k$
as a~direct summand. Indeed, we have the~following formula:

\begin{proposition}
	Let $X=\{x_0,\dots,x_{k+1}\}$ and $\sigma_l\colon X_0\to X_0$ be as in Example~\ref{ex:sigma-spindle}. Then
	\begin{equation}\label{eq:homology-of-sigma-spindle}
		H_1(X_{\sigma_k}) = \Z^2\oplus\Z_k.
	\end{equation}
	In particular, every finite cyclic group appears as the torsion of the first homology of some spindle.
\end{proposition}

\noindent
This proposition follows from a~more general result that holds for any $f$-spindle.

\begin{theorem}\label{thm:H1-for-f-spindle}
	The~first homology group $H_1(X_f)$ of an~$f$-spindle $X_f$ is generated by
	\begin{enumerate}
		\item pairs $(f(y),y)$, one per an~initial element $y\in X_0$,
		\item pairs $(v^i,b)$ and $(v^i,y)$, where $y\in X_0$ is not in the~same orbit as $v^i$, and
		\item sums $(b,c_1)+\ldots+(b,c_k)$, one for each cycle $(c_1,\ldots,c_k)$ in $\Gamma_f$,
	\end{enumerate}
	subject to a~relation $\ell \cdot (f(y),y)\equiv 0$. In particular,
	\begin{equation}\label{eq:H1-for-f-spindle}
		H_1(X_f) = \Z^{|X_0|(\orb(f)-1)+2\orb(f)}\oplus\Z_\ell^{\init(f)}
	\end{equation}
	if $X$ is a~finite set.
\end{theorem}

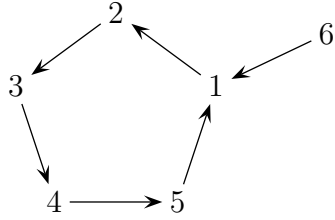
\begin{figure}[t]
	\psset{unit=0.7cm,nodesep=1mm,linewidth=0.5pt,arrowsize=5pt}%
	\begin{pspicture}(4.0,4.0)
		\rput(3.90,2.4){\rnode{x1}{1}}
		\rput(2.00,3.8){\rnode{x2}{2}}
		\rput(0.10,2.4){\rnode{x3}{3}}
		\rput(0.83,0.2){\rnode{x4}{4}}
		\rput(3.17,0.2){\rnode{x5}{5}}
		\rput(6.00,3.4){\rnode{x6}{6}}
		\ncline{->}{x1}{x2}\ncline{->}{x2}{x3}\ncline{->}{x3}{x4}
		\ncline{->}{x4}{x5}\ncline{->}{x5}{x1}\ncline{->}{x6}{x1}
	\end{pspicture}
	\caption{A~graph of the~function $\sigma_5$ from Example~\ref{ex:sigma-spindle}.}%
	\label{fig:sigma-graph}
\end{figure}

\begin{corollary}
	Every power of a~finite cyclic group can be realized as torsion of a~first homology for some spindle.
	Namely, let $X_0=\{1,\dots,k+r\}$ and define $\sigma_{k,r}\colon X_0\to X_0$ by the~formula
	\begin{equation}
		\sigma_{k,r}(n) := \begin{cases}
			n+1,&\textrm{if } n<k,\\
			1,&\textrm{if } n\geqslant k
		\end{cases}
	\end{equation}
	Then the~torsion subgroup $H_1(X_{\sigma_{k,r}})$ is isomorphic to $\Z_k^r$.
\end{corollary}

We need one technical, but useful, fact before we prove Theorem~\ref{thm:H1-for-f-spindle}. It will
be an~important tool for the~calculation of higher homology groups in the~next section.

\begin{lemma}\label{lem:compl-of-ker-diff}
	Choose $\underline y\in\CN_n(X)$ with $y_0\neq b$ and an~orbit $\Orb{a}$ of $a\in X_0$. Let $V\subset\CN_{n+2}(X)$
	and $W\subset\CN_{n+1}(X)$ be subgroups spanned by sequences $(b,x,\underline y)$ and $(x,\underline y)$
	respectively, with $x\in\Orb{a}$. If $\Orb{a}\neq\Orb{y_0}$ we also add $(f(y_0),\underline y)$
	to the~list of generators of $W$. The~restricted differential $\diff\colon V\longrightarrow W$ is
	injective and $\coker\diff$ is generated by $(a,\underline y)$, if $\Orb{a}\neq\Orb{y_0}$, and
	$(f(y_0),\underline y)$ subject to the~relation $k \cdot (f(y_0),\underline y)\equiv 0$, if $\Orb{y_0}$
	has a~cycle of length $k$.
\end{lemma}
\begin{proof}
	We will prove this lemma by computing the~quotient $Q:=\coker\diff/(f(y_0),\underline y)$.
	Each element $\diff(b,x,\underline y)$ gives a~relation in $Q$
	\begin{equation}\label{eq:Cbxx-rel}
		(x,\underline y)\equiv (f(x),\underline y).
	\end{equation}
	Hence, we can replace $x$ with any other element from its orbit. In~particular $Q=0$ if $y_0$ and
	$a$ are in the~same orbit. Otherwise, it is freely generated by $(a,\underline y)$. On the~other
	hand, the~kernel of the~composition
	\begin{equation}
		V\stackrel\diff\longrightarrow\coker\diff\longrightarrow Q
	\end{equation}
	is trivial, if the~orbit of $a$ is a~directed tree, and one-dimensional otherwise, generated by
	a~sum $(b,c_1,\underline y)+\ldots+(b,c_k,\underline y)$, where $(c_1,\ldots,c_k)$ is a~cycle
	in $\Orb{a}$. The~latter is mapped by $\diff$ to $k(f(y_0),\underline y)$. Hence, $\ker\diff=0$
	and the~cokernel is as expected.
\end{proof}

\begin{proof}[Proof of Theorem \ref{thm:H1-for-f-spindle}]
	Because for a~spindle we have $H_1(X)=\HN_1(X)$, we will consider only sequences without repetitions.
	The~first differential $\diff\colon\CN_1(X_f)\to\CN_0(X_f)$ is given by
	\begin{equation}
		\diff(x,y) = y - x\opn y =
				\begin{cases}
					0,			& \textrm{if }x\neq b,\\
					y-f(y),	& \textrm{if }x=b.
				\end{cases}
	\end{equation}
	Hence, the~kernel of $\diff$ is freely generated by
	\begin{itemize}
		\item pairs $(x,y)$ with $x\neq b$ and
		\item sums $(b,c_1)+\ldots+(b,c_k)$, where $(c_1,\ldots,c_k)$ is a~cycle in $\Gamma_f$.
	\end{itemize}
	Now consider relations introduced by $\diff(x,y,z)$.
	If $x,y\neq b$, then $\diff(x,y,z)=(z,z)=0$. When only $y\neq b$, the relations are
	\begin{align}
		(f(y),z)&\equiv (y,z) + (f(z),z),\quad\textrm{if }z\neq b, \textrm{ and}\\
		(f(y),b)&\equiv (y,b).
	\end{align}
	According to Lemma~\ref{lem:compl-of-ker-diff}, this restricts pairs $(x,y)$ to $(v^i,y)$, where
	$v^i$ and $y$ are from different orbits, and to $(f(y),y)$ (with $y\neq b$). The~latter is
	annihilated by the length of any cycle in the~graph $\Gamma_f$.

	If $y$ is initial, there are no more relations among generators $(x,y)$. Otherwise, for $y=f(z)$
	we have $\diff(x,b,z)=(z,f(z))=(z,y)$, which forces $(f(y),y)$ to be zero:
	\begin{equation}
		(f(y),y) \equiv (z,y) + (f(y),y) \equiv (f(z),y) = (y,y) \equiv 0.
	\end{equation}
	This fulfills all relations. In particular, each cycle $\underline c$ in $\Gamma$ contributes
	a~free generator to $\HN_1(X_f)$ and sequences $(f(y),y)$ have order $\ell$. This ends the~proof.
\end{proof}

\begin{corollary}
	First homology of an~$f$-spindle $X_f$ has torsion if an only if the~following three conditions
	hold:
	\begin{enumerate}
		\item $f$ has an~initial element,
		\item $f$ has a~cycle,
		\item length of cycles of $f$ are not co-prime, i.e. they have a~common divisor $d>1$.
	\end{enumerate}
	The~second condtition is automatic if $X$ is finite, but not the~others.
\end{corollary}

\section{Higher homology groups for \texorpdfstring{$f$}{f}-spindles}\label{sec:homology}
We will now compute higher homology groups for an~$f$-spindle and for simplicity we will restrict
to the~normalized part. Doing so already determines the~whole homology, as explained
in Theorem~\ref{thm:CD-from-CN} (see Corollary~\ref{cor:distr-homology-whole}).

In this section, $X$ will always stand for an~$f$-spindle induced by a~fixed function
\hbox{$f\colon\! X_0\!\to\! X_0$}, where $X=X_0\cup\{b\}$. Recall from the~previous section that each connected
component $\Gamma^0_f$ of the~graph $\Gamma_f$ is represented by some vertex $v^i$ and it is either
an~infinite directed tree or it contains a~unique cycle $\underline c = (c_1,\dots,c_k)$ of length $k$.
In particular, the~set of distinguished vertices $\{v^i\}$ parametrizes the~set of orbits in $X$
different from $\{b\}$. Finally, $\ell$ denotes the~greatest common divisor of lengths of all cycles
in $\Gamma_f$ (we set $\ell=0$ if $\Gamma_f$ has no cycles).

According to Theorem~\ref{thm:H1-for-f-spindle}, generators of $H_1(X)$ split into two groups:
sequences with two entries from the~same orbit or from two different orbits. The~first generate
the~torsion subgroup and the~latter are free. A~similar phenomenon occured in
Lemma~\ref{lem:compl-of-ker-diff}, where we compare orbits of the~first two entries in a~sequence.
This observation motivates the~following splitting of $\CN(X)$.

Let $\CNxx(X)$ be spanned by sequences $\underline x$ of length at least two, with $x_0$ and $x_1$
from the~same orbit. Clearly, for such a~sequence $d^j\underline x=0$ if $j\geqslant 2$ and
$d^0\underline x = d^1\underline x$. Hence, $\CNxx(X)$ is a~subcomplex of $\CN(X)$ and has
a~trivial differential. The~quotient complex $\CNxy(X):=\CN(X)/\CNxx(X)$ is freely spanned by sequences
$\underline x$ of length $1$ or with $x_0$ and $x_1$ lying in two different orbits (in particular,
we can take $b$ as one of them). Since $d^j\underline x\in\CNxx(X)$ for any sequence $\underline x$
as long as $j\geqslant 2$, the~differential in $\CNxy(X)$ has only two terms: $\diff=d^0-d^1$.

\begin{lemma}\label{lem:HNxy}
	The homology $\HNxy(X)$ is freely generated by three types of chains:
	\begin{itemize}
		\item type I:\phantom{II} $(v^i,x_1,\dots,x_n)$, where $x_1$ is in a~different orbit from that of $v^i$, and
		\item type II:\phantom{I} $(b,x_1,x_2,\dots,x_n)$, where both $x_1$ and $x_2$ are in the~same orbit, and
		\item type III: $\displaystyle{\sum_{i=1}^k(b,c_i,x_2,\dots,x_n)}$, where $(c_1,\dots,c_k)$ is a~cycle
					from beyond the~orbit of~$x_2$.
	\end{itemize}
	In all cases, neighboring entries are never equal.
\end{lemma}
\begin{proof}
	The~only case $\diff\underline x\neq 0$ is when $x_0=b$ and orbits of $x_1$ and $x_2$ are not
	the~same (or simply $\underline x = (b,x_1)$). In such a~case
	\begin{equation}\label{eq:HNxy-diff}
		\diff(b,x_1,\underline y) = (x_1,\underline y)-(f(x_1),\underline y).
	\end{equation}
	This has two consequences:
	\begin{enumerate}
		\item cycles are the~chains listed in the~lemma, except that in the~first case all sequences
					$\underline x$ with $x_0\neq b$ are allowed,
		\item boundaries \eqref{eq:HNxy-diff} only restrict type I generators: we can replace $x_0$ in
					$\underline x$ by any other element from the~same orbit; in particular by $v^i$. 
	\end{enumerate}
	This gives the~desired presentation of $\HNxy(X)$.
\end{proof}

The~chain complexes described above induce a~long exact sequence of homology
\begin{equation}\label{eq:LES-CNxx-CN-CNxy}
	\ldots\longrightarrow\CNxx_n(X)
				\longrightarrow\HN_n(X)
				\longrightarrow\HNxy_n(X)
				\stackrel{\delta_n}
				\longrightarrow\CNxx_{n-1}(X)
				\longrightarrow\ldots
\end{equation}
where $\delta_n([a])=\sum_{i=2}^n(-1)^id^na = \diff a$ is induced by the~full differential in
$\CN(X)$. Due to Lemma~\ref{lem:HNxy} the groups $\HNxy_n(X)$ are free, so are $\ker\delta_n$ which
results in a~splitting formula
\begin{equation}\label{eq:HN-ker-coker}
	\HN_n(X) \cong \ker\delta_n\oplus\coker\delta_{n+1}.
\end{equation}
It remains to compute both summands.

\begin{lemma}\label{lem:coker}
	The cokernel of $\delta_n$ is a~free $\Z_\ell$-module with basis consisting of sequences
	$(f(x),x,\dots)$ and $(f^2(x),f(x),x,\dots)$, where $x$ is initial in both cases.
\end{lemma}
\begin{proof}
	Since $\CNxx_n(X)=0$ for $n\leqslant 1$, $\coker\delta_n=0$ as well. This agrees with
	the~statement above, as there are no such sequences of length smaller than $2$. Hence, we will
	assume $n\geqslant 2$.

	According to Lemma~\ref{lem:compl-of-ker-diff}, the generators of $\HNxy_n(X)$ of the~second type
	are crucial: they are orthogonal to $\ker\delta_n$ and their images restrict generators of $\coker\delta_n$
	to sequences $(f(y),y,\ldots)$. Type~III generators, in turn, show that the length of any cycle in
	$\Gamma_f$ annihilates $\coker\delta_n$:
	\begin{equation}\label{eq:coker-rel3}
			0 \equiv \diff\left(\sum_{i=1}^k(b,c_i,x_2,\underline z)\right)
				= k(f(x_2),x_2,\underline z),
	\end{equation}
	so that it is a~$\Z_\ell$-module. To restrict the~set of generators even further, take a~type~I
	generator with $x_1=b$ and $x_2,x_3\in X_0$ (or just $x_2\in X_0$ if $n=2$). Then
	\begin{equation}\label{eq:coker-rel1}
			0 \equiv \diff(v^i,b,x_2,x_3,\underline z)
					= (x_2,f(x_2),x_3,\underline z) - (x_3,f(x_3),x_3,\underline z)
	\end{equation}
	makes it possible to replace $(f^2(x),f(x),y,\ldots)$ with $(f^2(y),f(y),y,\ldots)$, or to kill
	$(f^2(x),f(x))$ in case $n=2$, as we did in Theorem~\ref{thm:H1-for-f-spindle}. Also, $y$ must be
	initial --- otherwise, \eqref{eq:coker-rel1} forces $(f^2(y),f(y),y,\ldots)\equiv 0$, if we pick
	$x_3=y$ and $x_2$ such that $f(x_2)=y$.
	All the~remaining relations are induced by~sequences of the~form
	\begin{equation}	
		\underline x=(v^i,b,z_0,b,z_1,b,\ldots,b,z_k,z_{k+1},\ldots),
	\end{equation}
	perhaps ending at the~$b$ before $z_k$ or at $z_k$. Because $\diff\underline x$ is independent of
	$v^i$, we can choose one particular element. Then $\diff\underline x$ determines $\underline x$
	completely, so that all these boundaries are linearly independent. Each of them allows us to
	eliminate one more sequence from the~list of generators: $(z_0,f(z_0),b,\ldots)$ can be
	expressed as a~linear sum of sequences of type $(y,f(y),y,\ldots)\equiv(f^2(y),f(y),y,\ldots)$.
	This results in the~desired presentation of $\coker\delta_n$.
\end{proof}

If $X$ is finite, every component of $\Gamma_f$ must have a~cycle. Therefore, Lemma~\ref{lem:compl-of-ker-diff}
implies that $\delta_n$, when restricted to type~II generators, is an~isomorphism over $\Q$. Therefore
it is enough to count the~other generators to find the~rank of the~distributive homology of $X$.

\begin{theorem}\label{thm:Hn-for-f-spindle-finite}
	Assume $X$ is a~finite $f$-spindle. Then its homology is given by the~formulas
	\begin{equation}\label{eq:Hn-for-f-spindle-finite}
	\setlength\arraycolsep{2pt}
	\left\{\begin{array}{rl}
			\Caugm\HN_0(X) &= \Z^{\orb(f)},\\
						\HN_1(X) &= \Z^{(\orb(f)-1)|X_0|+2\orb(f)} \oplus
												\Z_\ell^{\init(f)}, \phantom{\Big|^{(|X|-1)^n}}\\
						\HN_n(X) &= \left(\Z^{(\orb(f)-1)|X|^2+|X|} \oplus
												\Z_\ell^{\init(f)|X|}\right)^{\oplus(|X|-1)^{n-2}},\quad\textrm{for }n\geqslant 2.
	\end{array}\right.
	\end{equation}
	In particular, $\HN_{n+1}(X) = \HN_n(X)^{\oplus(|X|-1)}$ for $n\geqslant 2$.
\end{theorem}
\begin{proof}
	Clearly, $\rk\Caugm\HNxy_0(X) = \orb(f)$, since the~only possible generators are $(v^i)$. For
	higher $n$, the~generators are counted in Tab.~\ref{tab:Hn-generators}. The~last two rows
	correspond to the~torsion part. Summing them up results in~formula \eqref{eq:Hn-for-f-spindle-finite}.
\end{proof}

\begin{table}
	\renewcommand{\arraystretch}{2}
	\begin{tabular}{rl|cr}
		\hline\hline
		\multicolumn{2}{c|}{Type of generators} & $n=1$ & \multicolumn{1}{c}{$n\geqslant 2$} \\
		\hline\hline
		$(v^i,x,\dots)$, & $x\in X_0$		& $(\orb(f)-1)|X_0|$ & $(\orb(f)-1)(|X|-1)^{n\phantom{-1}}$ \\
		$(v^i,b,\dots)$\phantom{,}	 &	& $\orb(f)$          & $\orb(f)(|X|-1)^{n-1}$ \\
		$\displaystyle{\sum_{i=1}^k(b,c_i,x,\dots)}$,& $\Orb{x}\neq\Orb{c_i},x\in X_0$
																		& $\orb(f)$ & $(\orb(f)-1)(|X|-1)^{n-1}$ \\
		$\displaystyle{\sum_{i=1}^k(b,c_i,b,\dots)}$\phantom{,}&
																		& $0$       & $\orb(f)(|X|-1)^{n-2}$ \\
		\hline
		$(f(y),y,\ldots)$,        & $y$ -- initial & $\init(f)$ & $\init(f)(|X|-1)^{n-1}$ \\
		$(f^2(y),f(y),y,\ldots)$, & $y$ -- initial & 0          & $\init(f)(|X|-1)^{n-2}$ \\
		\hline\hline
	\end{tabular}
	\vskip 0.5\baselineskip
	\caption{Numbers of generators in $\HN_n(X)$.}\label{tab:Hn-generators}
\end{table}

We can enhance the~theorem above by giving an~actual presentation of homology, including the~case
of infinite $f$-spindles. Indeed, since $\im\delta_n$ is a~free group, there is a~decomposition
$\HNxy_n(X) = \ker\delta_n\oplus V_n$ with $V_n\cong\im\delta_n$ and we can naturally identify
$\ker\delta_n$ with $\HNxy_n(X)/V_n$. To construct such a~$V_n$, we first assume $v^i$ belongs to
a~cycle, if its orbit has one, and we choose a~section $g\colon f(X_0)\to X_0$ of $f$. Furthermore,
if $\ell\neq 0$, we choose cycles $\underline c^1,\ldots,\underline c^r$ and nonzero numbers
$\alpha_1,\ldots,\alpha_r$ such that $\sum_{i=1}^r \alpha_ik^i = \ell$, where $k^i$ is the~length
of the~cycle $\underline c^i$. We then use the~chosen cycles to construct a~\emph{base cycle}
\begin{equation}
	\basecycle := \sum_{i=1}^r\alpha^i(c^i_1+\ldots+c^i_{k^i}).
\end{equation}
Notice, that $\diff(b,\basecycle,x_2,\ldots,x_n)	= \ell \cdot (f(x_2),x_2,\ldots x_n)$.

\begin{lemma}
	Fix an~element $v^0$ from among $v^i$'s and let $V_n\subset\HNxy_n(X)$ be generated by the~sequences
	\begin{enumerate}
		\item $(b,x_1,\ldots,x_n)$ with $\Orb{x_1}=\Orb{x_2}$, unless $x_1=v^i$ or $x_1=f(v^i)$,
					if already $x_2=v^i$, and
		\item $(v^0,b,g(y),x_3,\ldots,x_n)$ with $x_3=b$ or $y\neq f(x_3)$, and
		\item if $\Gamma_f$ has cycles, chains $(b,\basecycle,x_2,\ldots,x_n)$ with initial $x_2$ or
					$x_2=f(x_3)$ and initial $x_3$.
	\end{enumerate}
	Then $\delta_n|_{V_n}$ is injective and $\delta_n(V_n) = \im\delta_n$.
\end{lemma}
\begin{proof}
	Injectivity follows from Lemma~\ref{lem:compl-of-ker-diff} and carefully choosing the~other
	generators. Indeed, since we removed one sequence $(b,x_1,x_2,\ldots)$ for every cycle in
	the~orbit of $x_2$, the~quotient by the~first group of generators is freely generated by
	sequences $(f(y),y,\ldots)$. Then, as seen in the~proof of Lemma~\ref{lem:coker}, every
	sequence $(v^0,b,g(y),x_3,\ldots,x_n)$ lowers the~rank of the~cokernel by one and each chain from
	the~last group turns one of the~remaining generators into torsion of order $\ell$. This also
	shows $\delta_n(V_n)=\im\delta_n$.
\end{proof}

\begin{corollary}
	Let $X$ be an~$f$-spindle, not necessarily finite. Construct $V_n$ as above and choose a~cycle
	$\underline c^0$, if $\Gamma_f$ has one. Then the~generators of the~free part of $\HN_n(X)$
	are given modulo $V_n$ by the~following chains:
	\begin{enumerate}
		\item sequences $(b,f(v^i),v^i,x_3,\ldots,x_n)$ and $(b,v^i,x_2,\ldots,x_n)$ with $\Orb{v^i}=\Orb{x_2}$, and
		\item sequences $(v^i,x_1,\ldots,x_n)$, with $\Orb{v^i}\neq\Orb{x_1}$ and $x_1\neq b$, and
		\item sequences $(v^i,b,x_2,\ldots,x_n)$ with $v^i\neq v^0$ or $x_2\notin g(X_0')$, and
		\item sums $\displaystyle{\sum_{i=1}^k(b,c_i,x_2,\ldots,x_n)}$, one per cycle
					$(c_1,\ldots,c_k)$ from a~different orbit than $x_2$, except $\underline c^0$, when
					$x_2$ is initial or $x_2=f(x_3)$ and $x_3$ is initial.
	\end{enumerate}
	The~torsion subgroup\footnote{
		If $\ell=0$, these generators also contribute to the~free part and there is no torsion.
		In the~other extreme case $\ell=1$ the~torsion subgroup is trivial.
	} of $\HN_n(X)$ is a~$\Z_\ell$-module generated by sequences 
	$(f(y),y,x_2,\ldots,x_n)$ and $(f^2(y),f(y),y,x_3,\ldots,x_n)$, where $y$ is initial.
\end{corollary}

\noindent
If $X$ is finite, this presentation is coherent with Theorem~\ref{thm:Hn-for-f-spindle-finite}:
although we restrict free generators in the~last two groups, we include the~same number of
generators in the~first group that have not been counted before.

\section{Odds and ends}\label{sec:odds-ends}
\subsection*{Relative homology}
If $X$ is an~$f$-spindle, then $X_0=X-\{b\}$ is a~trivial spindle (i.e. $x\opn y = y$), so that
$\HN(X_0) = \CN(X_0)$. This makes it easy to compute the~relative homology $\HN(X,X_0)$. Indeed,
a~long exact sequence
\begin{equation}
	\ldots\longrightarrow\CN_n(X_0)
				\stackrel{i_n}
				\longrightarrow\HN_n(X)
				\longrightarrow\HN_n(X,X_0)
				\longrightarrow\CN_{n-1}(X_0)
				\stackrel{i_{n-1}}
				\longrightarrow\ldots
\end{equation}
implies $\HN_n(X,X_0)\cong \ker i_{n-1}\oplus \HN_n(X)/\im i_n$, since $\ker i_{n-1}$ is free.
Hence, we can obtain a~presentation for $\HN_n(X,X_0)$ as follows:
\begin{enumerate}
	\item Take a~presentation for $\HN_n(X)$.
	\item Remove generators $(v^i,\underline x)$ with $\underline x\in\CN_{n-1}(X_0)$.
				Notice that this kills both free and torsion generators.
	\item Add free generators coming from $\ker i_{n-1}$.
\end{enumerate}
Although this procedure results in a~presentation of relative homology, it misses a~very nice
structure of these groups. Every sequence from $\CN(X,X_0)$ can be written uniquely as
$(\underline x,b,\underline y)$, where each $y_i\neq b$ (both $\underline x$ and $\underline y$
might be empty). Because $b\opn y_i\neq b$, higher faces vanish so that in the~quotient
complex we have
\begin{equation}\label{eq:diff-CXX0}
	\diff(\underline x,b,\underline y) = \begin{cases}
			\phantom{(\diff\underline x)} 0, &\textrm{if }\underline x=\emptyset\textrm{ or }\underline x=(x_0),\\
			(\diff\underline x,b,\underline y)&\textrm{otherwise}.
	\end{cases}
\end{equation}
In particular, the~sequence $\underline y$ is preserved. This proves a~decomposition
\begin{equation}\label{eq:decomp-relative-into-b}
	\CN_{n+1}(X,X_0) = \bigoplus_{p+q=n}\Cb_p(X)\otimes\Caugm\CN_q(X_0),
\end{equation}
where $\Cb(X)$ is spanned by sequences ending with $b$. Notice that the~differential in $\Caugm\CN(X_0)$
is trivial, so the~formula above shows $\CN(X,X_0)$ is a~shifted total complex of the~bicomplex
$\Cb(X)\otimes\Caugm\CN(X_0)$.

To compute $\Hb_p(X)$ we note that the~normalized complex $\CN(X)$ splits into two copies
of $\Cb(X)$. Indeed, consider the~homomorphism $h\colon\CN(X)\to\CN(X)[1]$ given by
$h(\underline x) = (\underline x,b)$. It commutes with differentials\footnote{
	Recall that in the~shifted complex $C[1]_n=C_{n+1}$ and $\diff[1]_n = -\diff_{n+1}$ changes sign.
}
and $\Cb(X)=\ker h$. Moreover, the~image of $h$ is the~shifted reduced complex $\Caugm\Cb(X,b)[1]$,
because we can use $h$ to obtain all sequences except $(b)$. Finally, the~short exact sequence
\begin{equation}\label{eq:seq-Cb-CN-Cb}
	0	\longrightarrow\Cb(X)
		\longrightarrow\CN(X)
		\stackrel{h}\longrightarrow\Cb(X,b)[1]
		\longrightarrow 0
\end{equation}
splits via a~homomorphism $u\colon\Cb(X,b)[1]\to\CN(X)$ that forgets the~$b$ standing at the~end.
Hence, $\HN_n(X)\cong\Hb_n(X)\oplus\Caugm\Hb_{n+1}(X)$ and $\Hb_n(X)=\ker h_*$ is generated by
classes represented by sequences with $b$ at the~end. This, together with
\eqref{eq:decomp-relative-into-b}, results in another presentation for $\HN(X,X_0)$.

We finish this part by computing $\Hb(X)$ for a~finite $X$. This can be easily done using
the~split exact sequence~\eqref{eq:seq-Cb-CN-Cb} and Theorem~\ref{thm:Hn-for-f-spindle-finite}.

\begin{proposition}
	Assume $X$ is a~finite $f$-spindle. Then
	\begin{equation}
	\setlength\arraycolsep{2pt}
	\left\{\begin{array}{rl}
		\Hb_0(X) &= \Z,\\
		\Hb_1(X) &= \Z^{\orb(f)}, \phantom{\Big|^{(|X|-1)^n}}\\
		\Hb_n(X) &= \left(
												\Z^{\orb(f)|X|-|X|+1} \oplus
												\Z_\ell^{\init(f)}
										\right)^{\oplus(|X|-1)^{n-2}},\quad\textrm{for }n\geqslant 2.
	\end{array}\right.
	\end{equation}
\end{proposition}
\begin{proof}
	Clearly, $\Hb_0(X)=\Z$, generated by $(b)$. Directly from \eqref{eq:seq-Cb-CN-Cb} we compute that
	\begin{align}
		\rk\Hb_1(X) &= \rk\HN_0(X)-\rk\Hb_0(X) = \orb(f),							\textrm{ and}\\
		\rk\Hb_2(X) &= \rk\HN_1(X)-\rk\Hb_1(X) = \orb(f)|X|-(|X|-1).
	\end{align}
	$\HN_0(X)$ is free, so is $\Hb_1(X)$, and the~torsion subgroup of $\Hb_2(X)$ is equal to the~one
	of $\HN_1(X)$. For higher $n$ we use induction:
	\begin{equation}
		\begin{split}
			\rk\Hb_{n+3}(X) &= \rk\HN_{n+2}(X)-\rk\HN_{n+2}(X) \\
											&= (|X|-1)^n\big( (\orb(f)-1)|X|^2 + |X| - \orb(f)|X| + |X| - 1\big) \\
											&= (|X|-1)^n\big( \orb(f)|X|(|X|-1) - (|X|^2 - 2|X| + 1) \big) \\
											&= (|X|-1)^{n+1}\big( \orb(f)|X| - |X|+1 \big), \quad n\geqslant 0.
		\end{split}
	\end{equation}
	Torsion is even simplier to check.
\end{proof}

\subsection*{Realization of any finite abelian group}
We prove that every finite abelian group can be realized as the~torsion subgroup of $H_1(X)$ for
some spindle $X$. For this, we will first generalize Definition~\ref{def:f-spindle} to several
functions, see \cite{PrzSik-distr-hom}.

\begin{definition}\label{def:multi-f-spindle}
	Choose a~family of sets $\{X_i\}_{i\in I}$, not necessarily finite, and functions
	$f_i\colon X_i\to X_i$. Define the~spindle product on $X:=\coprod_{i\in I} X_i$
	for $x\in X_i$ and $y\in X_j$ by the~formula
	\begin{equation}
			x\opn y := \begin{cases}
												 y, &\textrm{if }i=j,\\
										f_j(y), &\textrm{if }i\neq j.
									\end{cases}
	\end{equation}
	Subsets $X_i\subset X$ are called \emph{blocks} of the~spindle $X$ and $f_i$'s are called
	\emph{block functions}. We will write $f\colon X\to X$ for the~function induced by all
	block functions.
\end{definition}

\begin{example}
Consider an~$f$-spindle which has two blocks, $X_0$ and $\{b\}$. The~block functions are given by
$f\colon X_0\to X_0$ and a~constant function on $\{b\}$.
\end{example}

From now on we assume $X$ has a~one-element block $\{b\}$. Then for every other block $X_i$,
the~sum $X_i^+:=X_i\sqcup\{b\}$ is an~$f_i$-spindle that is a~rectract of $X$, where the~retraction
$r\colon X\to X_i^+$ is the~identity on $X_i$ and maps everything else onto $b$. Hence, $\CN(X_i^+)$
is a~direct summand of $\CN(X)$.

\begin{theorem}\label{thm:all-finite-groups}
	Assume a~block spindle $X$ has a~one-element block $\{b\}$. Then
	\begin{equation}\label{eq:H1-more-torsion}
		H_1(X) \cong F\oplus\bigoplus_{i\in I}H_1(X_i^+),
	\end{equation}
	where $F$ is a~free abelian group of rank $\sum_{i\neq j} orb(f_i)|X_j| $. In~particular, every
	finite abelian group can be realized as the~torsion subgroup of $H_1(X)$ for some spindle $X$.
\end{theorem}
\begin{proof}
	We will assume there are at least two blocks different than $\{b\}$ --- otherwise the~statement
	is trivial. Since $H_1(X)=H_1(X,b)$, we will compute reduced homology. Each of $C(X_i^+,b)$ is
	still a~direct summand of $C(X,b)$, but now they have trivial intersections: no two of them
	have a~generator in common. This implies
	\begin{equation}
		C(X,b) \cong Q\oplus\bigoplus_{i\in I}C(X_i^+,b),
	\end{equation}
	where $Q$ is a~chain complex isomorphic to the~quotient of $C(X,b)$ by the~big direct sum. To
	compute $H_1(Q)$, we first notice that $Q_0 = 0$. Therefore, all $1$-chains are cycles and
	$H_1(Q) = \coker\diff$. Pick any sequence $(x,y,z)\in Q_2$. Its boundary is equal to
     \begin{equation}
		\diff(x,y,z) =\begin{cases}
											0, 							& \textrm{if $x$ and $y$ are from the same block},\\
											(y,z)-(f(y),z), & \textrm{otherwise}.
									\end{cases}
	\end{equation}
	The~induced relation only identifies some generators and does not introduce torsion. Namely, we
	can replace $(y,z)$ by any other pair $(y',z)$ with $y'$ from the same orbit as $y$. A~simple
	counting results in the~desired rank of $H_1(Q)$.	
\end{proof}

\begin{remark}
	Homology groups $H_n(Q)$ are usually not free when $n>1$, and the~same holds for their normalized
	versions $\HN_n(Q)$.
\end{remark}

\begin{remark}
	The method of this paper can be applied to the more general case of blocks spindles, even with no
	one-element block $\{b\}$. The proof is, however, much more involved and is postponed for
	future work.
\end{remark}

\subsection*{The degenerate part and growth conjectures}
We can easily compute the~distributive homology for an $f$-spindle $X$ using formula
\eqref{eq:CD-from-CN} from Theorem~\ref{thm:CD-from-CN}. Indeed, \eqref{eq:CD-from-CN} implies
the~following relation
\begin{equation}\label{eq:CD-recursive}
	\CD_{n+1}(X)\cong \CD_n(X)^{\oplus(|X|-1)} \oplus \widetilde C_{n-1}(X)^{\oplus|X|}
\end{equation}
and assuming $n\geqslant 2$, we can combine this with the~formula for the~normalized part from
Theorem~\ref{thm:Hn-for-f-spindle-finite} to obtain an~isomorphism of homology
\begin{equation}\label{eq:H-recursive-Fib}
	H_{n+1}(X) \cong H_n(X)^{\oplus(|X|-1)} \oplus H_{n-1}(X)^{\oplus|X|}, \textrm{ for }n\geqslant 2.
\end{equation}
In the case where $X$ is an~$f$-spindle, one has
\begin{equation}
	\begin{split}
		\rk H_2(X)
				&= \rk\HN_2(X) + \rk\HD_2(X)	\\
				&= \big((\orb(f)-1)|X|+1+\orb(f)\big)|X|	\\
				&= \big((\orb(f)-1)|X_0|+2\orb(f)\big)|X| = |X|\,\rk H_1(X),
	\end{split}
\end{equation}
which implies $H_1(X)^{\oplus|X|} \cong H_2(X)$, resulting in $H_{n+1}(X) \cong H_n(X)^{\oplus|X|}$
for $n\geqslant 1$.

\begin{corollary}\label{cor:distr-homology-whole}
	The~whole distributive homology for an~$f$-spindle $X$ is given by the~formulas
	\begin{equation}
	\setlength\arraycolsep{2pt}
	\left\{\begin{array}{rl}
		\widetilde H_0(X) &= \Z^{\orb(f)},\\
		           H_n(X) &= \left( \Z^{\orb(f)(|X|+1) - (|X|-1)} \oplus \Z_\ell^{\init(f)} \right)^{\oplus |X|^{(n-1)}},
					\quad\textrm{for }n\geqslant 1.
	\end{array}\right.
	\end{equation}
	In particular, $H_{n+1}(X)\cong H_n(X)^{\oplus |X|}$ for $n\geqslant 1$.
\end{corollary}

In \cite{PrzSik-distr-hom} the~following conjecture was stated:

\begin{conjecture}[Rank Growth Conjecture]\label{conj:RGC}
	Let $(X,\opn)$ be a~shelf. Then for $n\geqslant |X|-2$ one has $\rk H_{n+1}(X) = |X|\,\rk H_n(X)$.
\end{conjecture}

\noindent 
Using formula \eqref{eq:CD-from-CN} for the~degenerate subcomplex one can then show that the~rank
of the~normalized homology grows by a~factor of $|X|-1$, see \cite{PrzPut-degen}.

\begin{conjecture}[Normalized Rank Growth Conjecture]\label{cong:NRGC}
	Let $(X,\opn)$ be a~spindle. Then one has $\rk\HN_{n+1}(X) = (|X|-1)\rk\HN_n(X)$ for $n\geqslant |X|-1$.
\end{conjecture}

Conjecture~\ref{conj:RGC} implies Conjecture~\ref{cong:NRGC}, but not the~other way. Indeed, we
cannot expect more than formula~\eqref{eq:H-recursive-Fib}. Although the~authors do not know of 
any~example of a~spindle that does not satisfy Conjecture~\ref{conj:RGC}, there are spindles where
$H_{n+1}(X)\not\cong H_n(X)^{\oplus|X|}$ because of torsion.

\begin{example}
	Let $X=\{1,2,3,4\}$ and define $\opn\colon X\times X\to X$ by the~following table:
	\begin{center}
	\begin{tabular}{c|cccc}
		\raisebox{1pt}{$\opn$} & 1 & 2 & 3 & 4 \\
		\hline
		1 & 1 & 2 & 4 & 3 \\
		2 & 1 & 2 & 4 & 3 \\
		3 & 2 & 1 & 3 & 4 \\
		4 & 2 & 1 & 3 & 4
	\end{tabular}
	\end{center}
	Computation on a~computer resulted in the~following groups:
	\begin{align*}
		H_0(X) &= \Z^2,										& H_3(X) &= \Z^{32} \oplus \Z_2^{52},\\
		H_1(X) &= \Z^2 \oplus \Z_2^4,			& H_4(X) &= \Z^{128}\oplus \Z_2^{204},\\
		H_2(X) &= \Z^8 \oplus \Z_2^{12},	& H_5(X) &= \Z^{512}\oplus \Z_2^{820}.
	\end{align*}
	One can easily check that $H_n(X)\cong H_{n-1}(X)^{\oplus 3}\oplus H_{n-2}(X)^{\oplus 4}$ for
	$3\leqslant n\leqslant 5$ and that the~Rank Growth Conjecture holds. However, the~torsion
	subgroup does not grow by the~factor of $4$.
\end{example}

\noindent
This suggests the~following Growth Conjecture for distributive homology, including torsion.

\begin{conjecture}[Growth Conjecture]
	Let $(X,\opn)$ be a~shelf. Then for $n\geqslant |X|-2$ one has
	\begin{equation}
		H_{n+1}(X) \cong H_n(X)^{\oplus(|X|-1)}\oplus H_{n-1}(X)^{\oplus|X|}.
	\end{equation}
	Furthermore, if $X$ is a~spindle, then also $\HN_{n+1}(X)\cong\HN_n(X)^{\oplus(|X|-1)}$.
\end{conjecture}

\noindent
Theorem~\ref{thm:Hn-for-f-spindle-finite} shows $f$-spindles satisfy all of these conjectures. Also,
the~authors tested plenty of other block spindles in attempts to find a~counterexample to these
conjectures, but they did not succeed.

\subsection*{Acyclicity results}

Let $X$ be a~shelf and $A\subset X$ a~subset such that $X$ acts on $A$ from the~right by permutations,
i.e. $a\opn x\in A$ whenever $a\in A$ and the~map $a\mapsto a\opn x$ is a~permutation of $A$ for every
$x\in X$. If such an~$A$ exists and is finite, it was proved in~\cite{Prz-distr-survey} that $H(X)$ is
annihilated by $|A|$. It was expected to be trivial as one-term distributive homology was supposed to
be torsion-free. However, we have seen already the~latter is not true and it is no longer obvious why
homology groups of such a~spindle should vanish. We prove this below. To simplify notation, we will
omit $\opn$ and use the~left-first convention for bracketing:
\begin{equation}
	x_1\cdots x_n := ((x_1\opn x_2)\opn\cdots )\opn x_n.
\end{equation}
Distributivity of $\opn$ implies the~generalized distributivity:
$(x_1\cdots x_n)\opn y = (x_1\opn y)\cdots (x_n\opn y)$.

\begin{proposition}\label{prop:div-EX-acyclic}
	Let $(X,\opn)$ be a~shelf with a~subset $A\subset X$ on which $X$ acts from the~right by
	permutations. Then the~complex $\widetilde C(X)$ is acyclic.
\end{proposition}
\begin{proof}
	We will construct a~contracting homotopy $h\colon \widetilde C_n(X)\to \widetilde C_{n+1}(X)$.
	First, notice that for every element $a\in A$ and $x\in X$ we can find a~unique $a'\in A$ such
	that $a=a'\opn x$. More generally, for a~fixed $a\in A$ there is a~unique solution $a_{\underline x}$
	to an~equation $a = a_{\underline x}x_0\cdots x_n$ for any sequence $\underline x=(x_0,\ldots,x_n)$.
	Using the~distributivity of $\opn$ we can transform the~right hand side by moving $x_i$ to the~left,
	which results in the~equality
	\begin{equation}
		a = (a_{\underline x}\opn x_i)\cdots(x_{i-1}\opn x_i)x_{i+1}\cdots x_n.
	\end{equation}
	This means that $a_{\underline x}\opn x_i = a_{d^i\underline x}$ and the~map
	$h(\underline x) := (a_{\underline x},\underline x)$ satisfies
	\begin{equation}
		d^{i+1}h(\underline x) = (a_{\underline x}\opn x_i, d^i\underline x) = h(d^i\underline x)
	\end{equation}
	for every $0\leqslant i\leqslant n$. Hence, $\diff h(\underline x) + h(\diff\underline x) =
	d^0h(\underline x) = \underline x$ and the~identity homomorphism on $\widetilde C(X)$ is
	nullhomotopic.
\end{proof}

\end{document}